\documentclass[10pt]{amsart}
\usepackage{amsmath,amsthm,amsfonts,amssymb,latexsym,color}
\usepackage[colorlinks=true,linkcolor=blue,citecolor=blue,urlcolor=blue]{hyperref}
\usepackage{microtype}

\newcommand{\qbinom}[2]{\left[ \genfrac{}{}{0pt}{}{#1}{#2} \right]_q}
\newcommand{\TrZ}{\operatorname{Tr}_{\mathbb Z}}

\title[The $\frac{1}{2}$-Conjecture for $q$-Binomial Coefficients]{The $\frac{1}{2}$-Conjecture for $q$-Binomial Coefficients with
       Fractional Index}

\author{Guo-Niu HAN and Huan XIONG}
\address{Universit\'e de Strasbourg, CNRS, IRMA UMR 7501, F-67000 Strasbourg, France}
\email{guoniu.han@unistra.fr}
\address{Institute for Advanced Study in Mathematics, Harbin Institute of Technology, Harbin 150001, China}
\email{huan.xiong.math@gmail.com}

\date{August 7, 2026}

\newtheorem{theorem}{Theorem}[section]
\newtheorem{conjecture}[theorem]{Conjecture}
\newtheorem{proposition}[theorem]{Proposition}
\newtheorem{lemma}[theorem]{Lemma}

\theoremstyle{definition}
\newtheorem{definition}[theorem]{Definition}
\newtheorem{example}[theorem]{Example}
\theoremstyle{remark}
\newtheorem{remark}[theorem]{Remark}

\begin{document}

\begin{abstract}
For a nonnegative integer $k$ and a rational number $r\in\mathbb{Q}^+$,
we define the generalized Gaussian binomial coefficient
$\qbinom{r+k}{k} = \frac{(q^{r+1}; q)_k}{(q; q)_k}$.
When $r=a/b$ with $a,b$ coprime positive integers and $b\geq 2$,
expanding $\qbinom{r+k}{k}$ via the finite $q$-binomial theorem produces
fractional powers of $q$, so that $\qbinom{r+k}{k}$ is a \emph{Puiseux series}
in $q$ with nonnegative exponents; concretely it lies in $\mathbb{Q}[[q^{1/b}]]$.
The notion we single out is the \emph{integer trace} of this expansion, the
subseries consisting of those terms $c_r(d)\,q^d$ whose exponent $d$ is an integer,
with all fractional powers discarded. This projection is not standard, and there is
no a~priori reason for the surviving coefficients to behave coherently as $r$ varies.
Nonetheless, ordering the family by the coefficientwise partial order leads to the
\emph{$\tfrac{1}{2}$-Conjecture}: among all $r\in\mathbb{Q}^+$, the value
$r=\tfrac{1}{2}$ maximizes the integer trace, in the sense that the coefficients of
$\qbinom{1/2+k}{k}$ dominate those of $\qbinom{r+k}{k}$ coefficientwise for every $r$.
That so elementary a definition should single out $\tfrac{1}{2}$ this cleanly came as
a surprise to us. We prove the conjecture in several special cases and provide
further computational evidence.
\end{abstract}

\maketitle

\bigskip

\section{Introduction}

The \emph{Gaussian binomial coefficients}, also called \emph{$q$-binomial coefficients},
are central objects in combinatorics, number theory, and the theory of $q$-series (see, e.g.,
\cite{Andrews1998,GasperRahman1990,Stanley2011}).
For integers $n\geq k\geq 0$ they are defined by
\begin{equation}\label{eq:gaussian}
  \qbinom{n}{k}
  \;=\; \frac{(q;\,q)_n}{(q;\,q)_k\,(q;\,q)_{n-k}},
\end{equation}
where
$$
(u;q)_n :=
  \begin{cases}
    1, & \text{if } n = 0, \\
    (1-u)(1-uq) \cdots (1-uq^{n-1}), & \text{if } n \ge 1,
  \end{cases}
$$
is the standard \emph{$q$-shifted factorial} (or \emph{$q$-Pochhammer symbol})~\cite[p.~3]{GasperRahman1990}.
It is classical that $\qbinom{n}{k}$ is a polynomial in $q$ with nonnegative
integer coefficients, and that it counts the number of $k$-dimensional subspaces
of an $n$-dimensional vector space over $\mathbb{F}_q$
\cite[Proposition~1.7.2]{Stanley2011}.

By cancellation in~\eqref{eq:gaussian}, we may write
\begin{equation}\label{eq:shift-form}
  \qbinom{r+k}{k}
  \;=\; \frac{(1-q^{r+1})(1-q^{r+2})\cdots(1-q^{r+k})}{(1-q)(1-q^2)\cdots(1-q^k)}
\end{equation}
for nonnegative integers $r$ and $k$.  The right-hand side of~\eqref{eq:shift-form}
is meaningful for \emph{any} $r\in\mathbb{R}$, and in particular for rational $r$.
For example, setting $r=\tfrac{1}{2}$ gives
\[
  \qbinom{1/2+k}{k}
  \;=\;
  \frac{(1-q^{3/2})(1-q^{5/2})\cdots(1-q^{1/2+k})}
       {(1-q)(1-q^2)\cdots(1-q^k)},
\]
which is a well-defined power series in $q^{1/2}$.

\medskip

In this paper, we study the extension of~\eqref{eq:gaussian} to rational numbers.
Following the notation of~\eqref{eq:shift-form},
we study the family
\begin{equation}
	\qbinom{r+k}{k} \;=\;  \frac{(q^{r+1};\,q)_k}{(q;\,q)_k},
  \qquad r\in\mathbb{Q}^+,
\end{equation}
which coincides with the classical Gaussian binomial coefficient when $r$ is a positive integer.
For $r=a/b$ with $a,b$ coprime positive integers and $b\geq 2$,
expanding $(q^{r+1};\,q)_k$ via the finite $q$-binomial
theorem~\cite[Theorem~3.3]{Andrews1998} produces fractional powers of $q$, so
$\qbinom{r+k}{k}$ is a \emph{Puiseux series} in $q$ with nonnegative exponents;
concretely it lies in $\mathbb{Q}[[q^{1/b}]]$.
Recall that the \emph{ring of Puiseux series} (with nonnegative exponents) over
$\mathbb{Q}$ is
\begin{equation}\label{eq:puiseux-ring}
  \mathcal{P} \;:=\; \bigcup_{b\geq 1}\mathbb{Q}[[q^{1/b}]],
\end{equation}
the set of formal series in $q$ whose exponents are nonnegative rationals admitting
a common denominator.
As $r=a/b$ ranges over $\mathbb{Q}^+$ the denominator $b$ is unbounded, so no single
$\mathbb{Q}[[q^{1/b}]]$ contains the entire family $\{\qbinom{r+k}{k}\}_{r\in\mathbb{Q}^+}$;
the ring $\mathcal{P}$ provides a common ambient space for all of them.

\medskip
Once $\qbinom{r+k}{k}$ is viewed inside $\mathcal{P}$, a natural question is what to
do with its fractional powers. For
\[
  A(q)=\sum_{\alpha\in\mathbb{Q}_{\geq 0}}a_\alpha q^\alpha\in\mathcal P,
\]
where the exponents occurring in $A$ have a common denominator, define the
\emph{integer-trace projection} by
\begin{equation}\label{eq:trace-definition}
  \TrZ A(q):=\sum_{d\in\mathbb Z_{\geq0}}a_dq^d\in\mathbb Q[[q]].
\end{equation}
Thus $\TrZ$ discards every nonintegral power.  We also write
$A(q)\big|_{\mathbb Z}:=\TrZ A(q)$.

We use the usual coefficientwise partial order on all of $\mathbb Q[[q]]$, not
only on series whose coefficients are nonnegative.  For $F,G\in\mathcal P$, write
\begin{equation}\label{eq:trace-order-definition}
  F\geq_q G
  \quad\Longleftrightarrow\quad
  \TrZ(F-G)\in\mathbb Q_{\geq0}[[q]].
\end{equation}
Equivalently, every coefficient of $\TrZ F$ is at least the corresponding
coefficient of $\TrZ G$.  This is a preorder on $\mathcal P$ and descends to the
usual coefficientwise partial order on the integer traces in $\mathbb Q[[q]]$.
We write $F=_qG$ when $\TrZ F=\TrZ G$.
We are not aware of this projection having been studied before. Since the
integer-exponent terms arise from cancellations among fractional powers, one might
expect the surviving coefficients to carry little structure as $r$ varies.

The projection is linear but is not multiplicative in general.  We shall use the
following elementary compatibility whenever one factor has only integral exponents.

\begin{lemma}[Integral-factor rule]\label{lem:trace-factor}
For $A(q)\in\mathcal P$ and $B(q)\in\mathbb Q[[q]]$,
\begin{equation}\label{eq:trace-factor}
  \TrZ\bigl(A(q)B(q)\bigr)=\bigl(\TrZ A(q)\bigr)B(q).
\end{equation}
\end{lemma}

\begin{proof}
Multiplication by a monomial $q^j$ with $j\in\mathbb Z_{\geq0}$ preserves the
fractional part of every exponent of $A(q)$.  Consequently, no nonintegral
exponent of $A(q)$ can become integral after multiplication by a term of $B(q)$.
Taking the Cauchy product therefore gives~\eqref{eq:trace-factor} coefficientwise.
\end{proof}

Here is a small example showing both the support selection and the possible signs.

\begin{example}\label{ex:k4-traces}
Let $D_4(q)=(1-q)(1-q^2)(1-q^3)(1-q^4)$.  Direct expansion gives
\begin{align*}
 \TrZ\qbinom{1/2+4}{4}
   &=\frac{1+q^4+q^5+2q^6+q^7+q^8+q^{12}}{D_4(q)},\\
 \TrZ\qbinom{1/3+4}{4}
   &=\frac{1-q^7-q^8-q^9-q^{10}}{D_4(q)},\\
 \TrZ\qbinom{1/4+4}{4}
   &=\frac{1+q^{11}}{D_4(q)}.
\end{align*}
For denominator $2$, the contributing indices are $0,2,4$ and all signs are
positive.  For denominator $3$, the index $3$ contributes with a negative sign;
for denominator $4$, only $0$ and $4$ contribute, both positively but with
sparser support.  This is the basic pattern behind the conjecture.
\end{example}

\medskip
Extensive computation shows otherwise. Among all positive rational parameters $r$,
a single value, $r=\tfrac{1}{2}$, dominates every other in the coefficientwise
partial order on integer traces. We call this statement the
\emph{$\tfrac{1}{2}$-Conjecture} (Conjecture~\ref{conj:main}). It is striking that
so elementary an operation as deleting fractional powers should pick out
$\tfrac{1}{2}$ in this way.

\medskip
The paper is organized as follows.
Section~\ref{sec:conjecture} introduces the $\tfrac{1}{2}$-Conjecture
(Conjecture~\ref{conj:main}) and develops the key structural observation that all
integer-trace contributions from $\qbinom{1/2+k}{k}$ are nonnegative.
Section~\ref{sec:order-props} collects the basic properties of the integer trace
order $\geq_q$ used throughout the proofs.
Section~\ref{sec:smallk} proves the conjecture for small values $k=1,2,3$ by direct
computation.
Section~\ref{sec:partial} establishes the conjecture for the two infinite families
$r=m+\tfrac{1}{2}$ (half-integers) and $r=m$ (positive integers).
Finally, Section~\ref{sec:evidence} presents computer-verified evidence for $r=1/4$
and $r=1/3$, together with stronger auxiliary conjectures on the sign of the
numerator polynomials $H_k(q)$ and $U_k(q)$.

\section{The \texorpdfstring{$\tfrac{1}{2}$}{1/2}-conjecture}\label{sec:conjecture}
Let $k$ be a nonnegative integer.
We begin with the finite $q$-binomial theorem
(see, e.g., \cite[Theorem~3.3]{Andrews1998}):
\begin{equation}\label{eq:qbinom-thm}
  (u;\,q)_k \;=\; \sum_{s=0}^{k} (-1)^s \qbinom{k}{s} q^{\binom{s}{2}} u^s.
\end{equation}
Applying~\eqref{eq:qbinom-thm} with $u=q^{r+1}$ gives
\begin{equation}\label{eq:numerator}
  (q^{r+1};\,q)_k
  \;=\; \sum_{s=0}^{k} (-1)^s \qbinom{k}{s} q^{E(s,r)},
\end{equation}
where
\begin{equation}\label{eq:exponent}
  E(s,r) \;=\; \binom{s}{2} + (r+1)s
           \;=\; \frac{s(s+1)}{2} + rs.
\end{equation}

\begin{definition}\label{def:support}
For $r\in\mathbb{Q}^+$, the \emph{integer support} is
\[
  S_r \;:=\; \bigl\{ s\in\{0,1,\dots,k\} \;\big|\; E(s,r)\in\mathbb{Z} \bigr\}.
\]
\end{definition}

The integer trace of $\qbinom{r+k}{k}$ is determined entirely by the terms indexed by $S_r$,
since only those contribute integer powers of $q$ to the numerator.  More precisely,
the inverse denominator $(q;q)_k^{-1}$ belongs to $\mathbb Q[[q]]$, so the
integral-factor rule gives
\begin{equation}\label{eq:trace-through-denominator}
 \TrZ\qbinom{r+k}{k}
 =\frac{\TrZ\bigl((q^{r+1};q)_k\bigr)}{(q;q)_k}.
\end{equation}

\medskip

For example, when $r=1/2$,
the exponent formula~\eqref{eq:exponent} gives
\begin{equation}
  E\!\left(s,\tfrac{1}{2}\right)
  \;=\; \frac{s(s+1)}{2} + \frac{s}{2}
  \;=\; \frac{s^2+2s}{2}
  \;=\; \frac{s(s+2)}{2}.
\end{equation}
This is an integer if and only if $s(s+2)\equiv 0\pmod{2}$, which holds precisely
when $s$ is even.  Hence
\begin{equation}
	S_{1/2} \;=\; \{ 0, 2, 4, \ldots, 2\lfloor k/2\rfloor \}.
\end{equation}
Crucially, $(-1)^s=+1$ for every $s\in S_{1/2}$, so \emph{every contribution to
the integer trace of $\qbinom{1/2+k}{k}$ is nonnegative}.

\medskip

For the general case $r=a/b$ with $a,b$ coprime positive integers and $b\geq 2$,
since $s(s+1)/2$ is always an integer, the condition $E(s,a/b)\in\mathbb{Z}$
reduces to $as/b\in\mathbb{Z}$, and by $\gcd(a,b)=1$ this is equivalent to
$b\mid s$.  Hence
\begin{equation}
  S_{a/b} \;=\; \bigl\{0,\,b,\,2b,\,\ldots,\,b\lfloor k/b\rfloor\bigr\}.
\end{equation}
The integrality condition $b\mid s$ is periodic in $s$ with minimal period $b$.

Two regimes arise according to the parity of $b$:
\begin{itemize}
  \item If $b$ is odd (so $b\geq 3$), then $S_{a/b}$ contains odd values of
        $s$ (e.g.\ $s=b$), so the signs $(-1)^s$ in~\eqref{eq:numerator}
        alternate and the integer-trace numerator has both positive and
        negative contributions.
  \item If $b$ is even, then every $s\in S_{a/b}$ is a multiple of $b$ and
        hence even, so $(-1)^s=+1$ and the integer-trace numerator has
        nonnegative coefficients, just as for $r=\tfrac{1}{2}$.
\end{itemize}
Among all $b\geq 2$, the choice $b=2$ minimizes $b$ and so yields the
densest integer support: $S_{1/2}=\{0,2,4,\ldots\}\cap[0,k]$ has
$\lfloor k/2\rfloor+1$ elements, the maximum among all $S_{a/b}$.
This combination of nonnegative contributions and maximal support is what
motivates Conjecture~\ref{conj:main}.

\medskip

Write the integer trace expansion as
\begin{equation}
  \qbinom{r+k}{k}\big|_{\mathbb{Z}}
  \;=\; \sum_{d\geq 0} c_r(d)\,q^d.
\end{equation}
The denominator $(q;\,q)_k^{-1}$ has the well-known expansion
\cite[Theorem~1.1]{Andrews1998}:
\begin{equation}\label{eq:partition}
  \frac{1}{(q;\,q)_k}
  \;=\; \sum_{n\geq 0} p_k(n)\,q^n,
\end{equation}
where $p_k(n)\geq 0$ counts the number of partitions of $n$ into at most $k$ parts
(equivalently, into parts each at most $k$).
See also \cite[Eq. (1.76)]{Stanley2011}.

Combining~\eqref{eq:numerator} and~\eqref{eq:partition}, and applying
Lemma~\ref{lem:trace-factor} as in~\eqref{eq:trace-through-denominator}, yields the
convolution formula
\begin{equation}\label{eq:convolution}
  c_r(d)
  \;=\; \sum_{s\in S_r} (-1)^s
        \sum_{n\geq 0} \beta_{k,s}(n)\,p_k\!\bigl(d - n - E(s,r)\bigr),
\end{equation}
where $\beta_{k,s}(n)$ denotes the coefficient of $q^n$ in $\qbinom{k}{s}$.
Since $\qbinom{k}{s}$ is a polynomial in $q$ with nonnegative integer coefficients,
we have $\beta_{k,s}(n)\geq 0$ for all $n$, and likewise $p_k(n)\geq 0$.
The only possible source of negativity in~\eqref{eq:convolution} is therefore the
sign $(-1)^s$.

\medskip

This analysis leads to our main conjecture.

\medskip

\begin{conjecture}[The $\tfrac{1}{2}$-Conjecture]\label{conj:main}
For every $r\in\mathbb{Q}^+$ and every integer $k\geq 1$,
\begin{equation}
 \qbinom{1/2+k}{k} \geq_q \qbinom{r+k}{k}.
\end{equation}
\end{conjecture}

The conjecture asserts that $r=\tfrac{1}{2}$ is a \emph{global maximizer} of the
integer traces in the coefficientwise partial order on $\mathbb Q[[q]]$.

\medskip
The following special case deserves individual attention.

\begin{conjecture}[$\tfrac{1}{2}$ vs.\ $\tfrac{1}{4}$]\label{conj:1/4}
For all $k\geq 1$,
\[
 \qbinom{1/2+k}{k} \;\geq_q\;  \qbinom{1/4+k}{k} .
\]
\end{conjecture}

Conjecture~\ref{conj:1/4} is a special case of Conjecture~\ref{conj:main} with
$r=1/4$, but is singled out because of the proximity of $1/4$ to $1/2$ and because
the computation of both sides at the level of numerator polynomials
already exhibits rich combinatorial structure (see Section~\ref{sec:evidence}).

\section{Properties of the integer trace order}\label{sec:order-props}

The following properties of the integer trace order $\geq_q$ on $\mathcal P$ are
elementary but important.  We record them for use in the proofs of
Sections~\ref{sec:partial} and~\ref{sec:evidence}.

\begin{lemma}\label{lem:order-props}
For series in $\mathcal P$, the integer trace order satisfies the following properties.
\begin{enumerate}
	\item\label{it:equiv}  $F(q)\geq_q G(q)$ if and only if  $F(q)-G(q) \geq_q 0$.
	\item If $F(q)\geq_q G(q)$ and $c\in\mathbb{Q}^+$, then
		$cF(q)\geq_q cG(q)$.
	\item\label{it:add} If $F_1(q)\geq_q G_1(q)$ and $F_2(q)\geq_q G_2(q)$, then
		$F_1(q)+F_2(q)\geq_q G_1(q)+G_2(q)$.
	\item\label{it:transi} If $F_1(q)\geq_q F_2(q)$ and $F_2(q)\geq_q F_3(q)$, then
		$F_1(q)\geq_q F_3(q)$.
\item\label{it:denom} If $F(q)\geq_q G(q)$ and $k_1,k_2,\dots,k_\ell$ are positive integers,
  then
  \[
    \frac{F(q)}{(1-q^{k_1})(1-q^{k_2})\cdots(1-q^{k_\ell})}\;\geq_q\;
    \frac{G(q)}{(1-q^{k_1})(1-q^{k_2})\cdots(1-q^{k_\ell})}.
  \]
\item\label{it:diff} If $n,m\in\mathbb{N}$ with $n\leq m$, then
  $\dfrac{q^n-q^m}{1-q}\geq_q 0$.

\end{enumerate}
\end{lemma}

\begin{proof}
Properties \eqref{it:equiv}--\eqref{it:transi} follow directly from the
coefficientwise order on the traces in $\mathbb Q[[q]]$.  Notice in particular
that property~\eqref{it:add} compares the sum with $G_1+G_2$.

	\eqref{it:denom} Put
\[
 B(q)=\prod_{i=1}^{\ell}\frac{1}{1-q^{k_i}}\in\mathbb Q_{\geq0}[[q]].
\]
By Lemma~\ref{lem:trace-factor},
\[
 \TrZ\bigl((F-G)B\bigr)=\TrZ(F-G)B.
\]
If $F\geq_qG$, then both factors on the right have nonnegative coefficients,
and so does their Cauchy product.  Property~\eqref{it:equiv} now gives the claim.

\eqref{it:diff} We compute
\[
  \frac{q^n - q^m}{1-q} \;=\; q^n + q^{n+1} + \cdots + q^{m-1} \geq_q 0.
\]
\end{proof}

\section{Partial results: small \texorpdfstring{$k$}{k}}\label{sec:smallk}
In this section, we prove the $\tfrac{1}{2}$-conjecture when $k$ is small.

\begin{proposition}\label{prop:smallk}
	The $\tfrac{1}{2}$-conjecture is true for $k=1,2,3$.
\end{proposition}
\begin{proof}

For $k=1$ and any $r\in\mathbb{Q}^+$, we compare $\qbinom{1/2+k}{k}$ and $\qbinom{r+k}{k}$ directly:
\begin{align*}
\qbinom{1/2+k}{k} - \qbinom{r+k}{k}
  &\;=\; \frac{(1-q^{3/2})-(1-q^{r+1})}{1-q}
   \;=\; \frac{q^{r+1}-q^{3/2}}{1-q}
   \;=_q\; \frac{q^{r+1}}{1-q}.
\end{align*}
Now
$$
	q^{r+1} \;=_q\; \begin{cases}
		0,  &\text{if $r\notin \mathbb{N}$}, \\
		q^{r+1}, &\text{if $r\in \mathbb{N}$},
	\end{cases}
$$
so $q^{r+1} \geq_q 0$, and therefore $q^{r+1}/(1-q) \geq_q 0$ by Lemma~\ref{lem:order-props}\eqref{it:denom}.
\medskip

For $k=2$ and $r\in\mathbb{Q}^+$,
\begin{align*}
\qbinom{1/2+k}{k} - \qbinom{r+k}{k}
  &\;=\;
  \frac{(1-q^{3/2})(1-q^{5/2}) - (1-q^{r+1})(1-q^{r+2})}{(1-q)(1-q^2)}\\
	&= \frac{(q^{r+1}+q^{r+2}+q^4)-(q^{3/2}+q^{5/2}+q^{2r+3})}{(1-q)(1-q^{2})}\\
	&=_q \frac{q^{r+1}+q^{r+2}+q^4-q^{2r+3}}{(1-q)(1-q^{2})}.
\end{align*}
Since
$$
q^{r+1}+q^{r+2}+q^4-q^{2r+3}
\;=_q\;
	\begin{cases}
		q^4, &\text{if $r\notin \frac{1}{2}\mathbb{N}$},\\
		q^4 - q^{2r+3}, &\text{if $r\in \frac{1}{2}\mathbb{N}$ and $r\notin \mathbb{N}$},\\
		(q^{r+1}-q^{2r+3})+q^{r+2}+q^4, &\text{if $r\in \mathbb{N}$},
	\end{cases}
$$
we obtain in each case
$$
\frac{q^{r+1}+q^{r+2}+q^4-q^{2r+3}}{(1-q)(1-q^{2})}\;\geq_q\; 0
$$
by Lemma~\ref{lem:order-props}\eqref{it:diff}.
\medskip

For $k=3$ and $r\in\mathbb{Q}^+$,
\begin{align*}
& \qbinom{1/2+k}{k} - \qbinom{r+k}{k} \\
&\;=\;
	\frac{(1-q^{3/2})(1-q^{5/2})(1-q^{7/2}) - (1-q^{r+1})(1-q^{r+2})(1-q^{r+3})}{(1-q)(1-q^2)(1-q^3)}\\
&\;=\; \frac{V}{(1-q)(1-q^{2})(1-q^3)},
\end{align*}
where
$$
V \;=_q\; q^4+q^5+q^6+q^{r+1}+q^{r+2}+q^{r+3}-q^{2r+3}-q^{2r+4}-q^{2r+5}+q^{3r+6}.
$$

If $r\notin \frac{1}{2}\mathbb{N}$ and $r\notin \frac{1}{3}\mathbb{N}$, then
$V =_q q^4+q^5+q^6 \geq_q 0$.

If $r\in \frac{1}{2}\mathbb{N}$ and $r\notin \mathbb{N}$, then
$V =_q (q^4-q^{2r+3})+(q^5-q^{2r+4})+(q^6-q^{2r+5})$,
so $V/(1-q)\geq_q 0$ by Lemma~\ref{lem:order-props}\eqref{it:diff}.

If $r\in \frac{1}{3}\mathbb{N}$ and $r\notin \mathbb{N}$, then
$V =_q q^4+q^5+q^6+q^{3r+6} \geq_q 0$.

If $r\in \mathbb{N}$, then
$$
V \;=_q\; q^4+q^5+q^6+(q^{r+1}-q^{2r+3})+(q^{r+2}-q^{2r+4})+(q^{r+3}-q^{2r+5})+q^{3r+6},
$$
so $V/(1-q)\geq_q 0$ by Lemma~\ref{lem:order-props}\eqref{it:diff}.
\end{proof}

\section{Partial results: half-integers and positive integers}\label{sec:partial}

Using Lemma~\ref{lem:order-props}, we prove Conjecture~\ref{conj:main}
in the cases where $r$ is a half-integer $m+\tfrac{1}{2}$ ($m\geq 0$)
or a positive integer $m\geq 1$.

\begin{proposition}\label{prop:half-int}
The $\frac{1}{2}$-conjecture holds for half-integer parameters: for every integer $m\geq 0$,
\begin{equation}\label{eq:half-int}
 \qbinom{1/2+k}{k} \;\geq_q\; \qbinom{m+1/2+k}{k}.
\end{equation}
\end{proposition}

\begin{proof}
We prove~\eqref{eq:half-int} by induction on $m$.  The base case $m=0$ is trivial.
For the inductive step it suffices to show
\begin{equation}
 \qbinom{m+1/2+k}{k} \;\geq_q\; \qbinom{m+3/2+k}{k}.
\end{equation}

Set $r=m+\tfrac{1}{2}$.  We compute the difference $\qbinom{r+k}{k}-\qbinom{r+1+k}{k}$ directly:
\begin{align}
  & \qbinom{r+k}{k}- \qbinom{r+1+k}{k} \notag \\
  &\;=\;
  \frac{(1-q^{r+2})\cdots(1-q^{r+k})}{(1-q)(1-q^2)\cdots(1-q^k)}
  \left[(1-q^{r+1})-(1-q^{r+1+k})\right] \notag \\
  &\;=\;
  \frac{(1-q^{r+2})\cdots(1-q^{r+k})\cdot\bigl(q^{r+1+k}-q^{r+1}\bigr)}
       {(1-q)(1-q^2)\cdots(1-q^k)} \notag \\
  &\;=\;
	\frac{(-q^{r+1})\,(1-q^{r+2})\cdots(1-q^{r+k})}
       {(1-q)(1-q^2)\cdots(1-q^{k-1})}. \label{eq:diff-half}
\end{align}
We claim that the right-hand side of~\eqref{eq:diff-half} is $\geq_q 0$.

Since $r=m+\tfrac{1}{2}$ is a half-integer, the exponents $r+1,r+2,\dots,r+k$
are half-integers as well.  Expand
$(-q^{r+1})(1-q^{r+2})\cdots(1-q^{r+k})$ by choosing a subset
$J\subseteq\{2,3,\dots,k\}$ of the nonconstant terms.  If $|J|=j$, the
corresponding monomial has sign $(-1)^{j+1}$ and exponent
\[
  (r+1)+\sum_{i\in J}(r+i)
  =(j+1)r+I_J,
  \qquad I_J:=1+\sum_{i\in J}i\in\mathbb Z.
\]
Because $r=m+\tfrac{1}{2}$, this exponent is integral exactly when $j+1$ is
even (equivalently, $j$ is odd).  Every monomial that survives the integer-trace
projection therefore has sign $(-1)^{j+1}=+1$.  After collecting equal powers,
all coefficients in the integer trace remain nonnegative.
Therefore
$$
(-q^{r+1})(1-q^{r+2})\cdots(1-q^{r+k}) \;\geq_q\; 0,
$$
and Lemma~\ref{lem:order-props}\eqref{it:denom}, applied with denominator
$(1-q)(1-q^2)\cdots(1-q^{k-1})$, yields
$\qbinom{r+k}{k}-\qbinom{r+1+k}{k} \geq_q 0$.

By transitivity and induction, $\qbinom{1/2+k}{k}\geq_q \qbinom{m+1/2+k}{k}$ for all $m\geq 0$.
\end{proof}

\begin{proposition}\label{prop:int}
The $\frac{1}{2}$-conjecture holds for positive integer parameters: for every positive integer $m$,
\begin{equation}\label{eq:int}
 \qbinom{1/2+k}{k} \;\geq_q\; \qbinom{m+k}{k}.
\end{equation}
\end{proposition}

\begin{proof}
We use the integer partition interpretations of the two sides.

\medskip
\textit{Step 1: A lower bound for $\qbinom{1/2+k}{k}$.}
By definition,
\[
  \qbinom{1/2+k}{k}
  \;=\;
  \frac{(1-q^{3/2})(1-q^{5/2})\cdots(1-q^{1/2+k})}{(1-q)(1-q^2)\cdots(1-q^k)}.
\]
To extract the integer trace, we expand the numerator using~\eqref{eq:numerator}
with $r=\tfrac{1}{2}$:
\[
  (1-q^{3/2})\cdots(1-q^{1/2+k})
  \;=\; \sum_{s=0}^{k}(-1)^s q^{s(s+2)/2}\qbinom{k}{s}.
\]
Only even values of $s$ contribute integer powers, giving
\[
  (1-q^{3/2})\cdots(1-q^{1/2+k})\bigr|_{\mathbb{Z}}
  \;=\; \sum_{\ell=0}^{\lfloor k/2\rfloor} q^{2\ell(\ell+1)}\qbinom{k}{2\ell}
  \;=\; 1 + \sum_{\ell=1}^{\lfloor k/2\rfloor} q^{2\ell(\ell+1)}\qbinom{k}{2\ell}
  \;\geq_q\; 1.
\]
Therefore, by Lemma~\ref{lem:order-props}\eqref{it:denom},
\begin{equation}\label{eq:lower-half}
 \TrZ\qbinom{1/2+k}{k}
  \;=\;
  \frac{(1-q^{3/2})\cdots(1-q^{1/2+k})\bigr|_{\mathbb{Z}}}
       {(q;\,q)_k}
  \;\geq_q\;
  \frac{1}{(q;\,q)_k}.
\end{equation}
Here the equality is precisely Lemma~\ref{lem:trace-factor}, applied to the
inverse denominator $(q;q)_k^{-1}\in\mathbb Q[[q]]$.

\medskip
\textit{Step 2: An upper bound for $\qbinom{m+k}{k}$.}
By the classical combinatorial interpretation of the $q$-binomial coefficient
\cite[Theorem~1.1]{Andrews1998},
\begin{equation}
  \frac{1}{(q;\,q)_k}
  \;=\; \sum_{\lambda\in P_k} q^{|\lambda|},
\end{equation}
where $P_k$ is the set of all partitions whose parts are each at most $k$.
On the other hand,
\begin{equation}
  \qbinom{m+k}{k}
  \;=\; \sum_{\lambda\in P_{m,k}} q^{|\lambda|},
\end{equation}
where $P_{m,k}$ is the set of partitions with at most $m$ parts, each of size at most $k$.

Since $P_{m,k}\subseteq P_k$, we have
\begin{equation}\label{eq:subset}
  \frac{1}{(q;\,q)_k} \;\geq_q\; \qbinom{m+k}{k}.
\end{equation}

\medskip
Combining \eqref{eq:lower-half} and~\eqref{eq:subset} yields \eqref{eq:int}.
\end{proof}

\section{Further examples}\label{sec:evidence}
We verify the $\frac{1}{2}$-conjecture for $r=1/3$ and $r=1/4$ in the range $1\leq k\leq 150$
with the help of a computer.

The following observation makes every computer check in this section a finite,
directly inspectable certificate.

\begin{lemma}[Prefix-sum certificate]\label{lem:prefix-sums}
Let $A(q)=\sum_{e=0}^{D}a_eq^e\in\mathbb Q[q]$.  Then
\begin{equation}\label{eq:prefix-sums}
  \frac{A(q)}{1-q}\geq_q0
  \quad\Longleftrightarrow\quad
  \sum_{e=0}^{d}a_e\geq0\quad\text{for every }0\leq d\leq D.
\end{equation}
If these equivalent conditions hold, then for every $k\geq1$,
\begin{equation}\label{eq:prefix-full-denominator}
  \frac{A(q)}{(q;q)_k}\geq_q0.
\end{equation}
\end{lemma}

\begin{proof}
For every $d\geq0$,
\[
 [q^d]\frac{A(q)}{1-q}=\sum_{e=0}^{\min(d,D)}a_e.
\]
For $d\geq D$ this coefficient is the total sum $\sum_{e=0}^{D}a_e$, so it is
enough to check the $D+1$ prefix sums in~\eqref{eq:prefix-sums}.  Finally,
\[
 \frac{A(q)}{(q;q)_k}=\frac{A(q)}{1-q}
 \prod_{j=2}^{k}\frac{1}{1-q^j},
\]
and every remaining factor has nonnegative coefficients.
\end{proof}

\begin{proposition}\label{prop:deno4}
For every integer $k$ with $1\leq k\leq 150$,
\begin{equation}
	\qbinom{1/2+k}{k} \;\geq_q\; \qbinom{1/4+k}{k}.
\end{equation}
\end{proposition}

\begin{proof}
By identity~\eqref{eq:numerator} with $r=1/2$ and $r=1/4$,
\begin{align*}
	&	\qbinom{1/2+k}{k} - \qbinom{1/4+k}{k} \\
	&= \frac{\displaystyle\sum_{j=0}^k(-1)^j q^{j/2+j(j+1)/2}\qbinom{k}{j}
        - \sum_{j=0}^k(-1)^j q^{j/4+j(j+1)/2}\qbinom{k}{j}}
       {(1-q)\cdots(1-q^k)}\\
	&=_q \frac{H_k(q)}{(1-q)(1-q^2)\cdots(1-q^k)},
\end{align*}
where
\begin{equation}\label{eq:Hk}
  H_k(q)
  \;=\;
  \sum_{\ell=1}^{\lfloor k/2\rfloor} q^{2\ell(\ell+1)}\qbinom{k}{2\ell}
  - \sum_{\ell=1}^{\lfloor k/4\rfloor} q^{\ell(8\ell+3)}\qbinom{k}{4\ell}.
\end{equation}

We list the values of $H_k(q)$ for $k=1,2,3,4$:
\begin{align*}
		H_1(q)&=0;\\
		H_2(q)&=q^4;\\
		H_3(q)&=q^4+q^5+q^6;\\
		H_4(q)&=q^4+q^5+2q^6+q^7+q^8-q^{11}+q^{12}.
\end{align*}
In the first three cases $H_k(q)$ has nonnegative coefficients, so
$\qbinom{1/2+k}{k} - \qbinom{1/4+k}{k} \geq_q 0$ for $k=1,2,3$
by Lemma~\ref{lem:order-props}\eqref{it:denom}.

For $k=4$, the polynomial $H_4(q)$ already contains the term $-q^{11}$, so
$H_4(q)$ itself is not coefficientwise nonnegative.  Nevertheless, its successive
coefficient sums are all nonnegative (the coefficient $-1$ in degree $11$ is
absorbed by the preceding positive mass).  Lemma~\ref{lem:prefix-sums} therefore
proves the required inequality for $k=4$.

For every $1\leq k\leq150$, our symbolic computation expands $H_k(q)$ with exact
integer coefficients and checks the finite certificate
\begin{equation}\label{eq:H-prefix-certificate}
  \sum_{e=0}^{d}[q^e]H_k(q)\geq0
  \qquad(0\leq d\leq\deg H_k).
\end{equation}
Lemma~\ref{lem:prefix-sums} converts~\eqref{eq:H-prefix-certificate} directly into
$H_k(q)/(q;q)_k\geq_q0$, which proves the proposition.  Thus the computation does
not rely on a heuristic pairing: the recorded prefix sums are an exact finite
certificate for every $k$ in the stated range.
\end{proof}

The corresponding code and data are available on the first author's personal webpage:
\begin{center}
\url{https://irma.math.unistra.fr/~guoniu/halfconj.html}
\end{center}
This allows for independent verification of the computations.

\begin{proposition}\label{prop:deno3}
For every integer $k$ with $1\leq k\leq 150$,
\begin{equation}
	\qbinom{1/2+k}{k} \;\geq_q\; \qbinom{1/3+k}{k}.
\end{equation}
\end{proposition}

\begin{proof}
By identity~\eqref{eq:numerator} with $r=1/2$ and $r=1/3$,
\begin{align*}
	&	\qbinom{1/2+k}{k} - \qbinom{1/3+k}{k} \\
	& = \frac{\displaystyle\sum_{j=0}^k(-1)^j q^{j/2+j(j+1)/2}\qbinom{k}{j}
        - \sum_{j=0}^k(-1)^j q^{j/3+j(j+1)/2}\qbinom{k}{j}}
       {(1-q)(1-q^2)\cdots(1-q^k)}\\
	&=_q \frac{U_k(q)}{(1-q)(1-q^2)\cdots(1-q^k)},
\end{align*}
where
\begin{equation}\label{eq:Uk}
  U_k(q)
  \;=\;
  \sum_{\ell=1}^{\lfloor k/2\rfloor} q^{2\ell(\ell+1)}\qbinom{k}{2\ell}
  - \sum_{\ell=1}^{\lfloor k/3\rfloor} (-1)^\ell q^{\ell(9\ell+5)/2}\qbinom{k}{3\ell}.
\end{equation}

We list the values of $U_k(q)$ for $k=1,2,3,4$:
\begin{align*}
U_1(q)&=0;\\
U_2(q)&=q^4;\\
U_3(q)&=q^4+q^5+q^6+q^7;\\
U_4(q)&=q^4+q^5+2q^6+2q^7+2q^8+q^9+q^{10}+q^{12}.
\end{align*}
As in the proof of Proposition~\ref{prop:deno4}, the program uses exact integer
arithmetic to verify
\begin{equation}\label{eq:U-prefix-certificate}
  \sum_{e=0}^{d}[q^e]U_k(q)\geq0
  \qquad(0\leq d\leq\deg U_k)
\end{equation}
for every $1\leq k\leq150$.  Lemma~\ref{lem:prefix-sums} then gives
$U_k(q)/(q;q)_k\geq_q0$, proving the proposition.  The code and data are
available on the same webpage.
\end{proof}

We next explain the lower limit $k=19$ in the two stronger conjectures below.

\begin{remark}[The observed threshold $19$]\label{rem:threshold-19}
The stronger property in question is coefficientwise nonnegativity of the
\emph{numerator polynomial itself}, rather than nonnegativity only after division
by $1-q$.  For $1\leq k\leq18$, exact computation gives
\[
\begin{array}{c|c}
\text{polynomial} & \text{values of $k\leq18$ having a negative coefficient}\\ \hline
H_k(q) & 4,5,6,\ldots,18\\
U_k(q) & 6,8,10,12,18.
\end{array}
\]
Both $H_k(q)$ and $U_k(q)$ are coefficientwise nonnegative for every tested value
$19\leq k\leq150$.  Thus $19$ is the smallest observed point after which neither
sequence has another failure in the computed range.  This is an empirical threshold,
not a consequence of the prefix-sum proof: even at the exceptional smaller values,
the prefix sums in~\eqref{eq:H-prefix-certificate} and
\eqref{eq:U-prefix-certificate} are nonnegative, which is enough for the two
propositions.  The conjectures assert that the stronger numerator positivity seen
from $k=19$ onward persists for all larger $k$.
\end{remark}

Empirical evidence suggests the following two stronger statements.

\begin{conjecture}\label{conj:1/4:nu}
	For all $k\geq 19$, we have $H_k(q)\geq_q 0$, where $H_k(q)$ is defined by~\eqref{eq:Hk}.
\end{conjecture}

\begin{conjecture}\label{conj:1/3:nu}
	For all $k\geq 19$, we have $U_k(q)\geq_q 0$, where $U_k(q)$ is defined by~\eqref{eq:Uk}.
\end{conjecture}


\end{document}